\documentclass[a4paper,12pt]{article}

\usepackage{amsmath,amsfonts,amssymb}
\usepackage{amsthm}
\usepackage{color,multicol}

\usepackage{hyperref}

\newtheorem{lemma}{Lemma}[section]
\newtheorem{theorem}[lemma]{Theorem}

\newtheorem{definition}[lemma]{Definition}

\newtheorem{remark}[lemma]{Remark}

\newcommand{\Dom}{\mathrm{Dom}}
\newcommand{\Spec}{\mathrm{Spec}}

\newcommand{\Xfont}{\mathcal{X}}
\newcommand{\Yfont}{\mathcal{Y}}

\begin{document}

\title{Level sets of the resolvent norm of a linear
operator revisited\footnote{MSC2010 subject classification 47A10, 47D06}}

\author{E. B. Davies\footnote{E-mail: \ E.Brian.Davies@kcl.ac.uk} \  and
 Eugene Shargorodsky\footnote{E-mail: \ eugene.shargorodsky@kcl.ac.uk} \\
Department of Mathematics, King's College London,\\
Strand, London WC2R 2LS, UK}

\date{}

\maketitle

\begin{abstract}
It is proved that the resolvent norm of an operator  with a compact
resolvent on a Banach space $X$ cannot be constant on an open set if
the underlying space or its dual is complex strictly convex. It is
also shown that this is not the case for an arbitrary Banach space: there
exists a separable, reflexive space $X$ and an unbounded, densely
defined operator acting in $X$ with a compact resolvent whose norm
is constant in a neighbourhood of zero; moreover $X$ is isometric to
a Hilbert space on a subspace of co-dimension $2$. There is also a bounded linear
operator acting on the same space whose resolvent norm is constant in a neighbourhood of zero.
It is shown that similar examples cannot exist in the co-dimension $1$ case.
\end{abstract}

\section{Introduction}

The $\varepsilon$--pseudospectrum of a closed densely defined linear
operator $A$ on a Banach space $X$  is usually defined as
\begin{equation}\label{strict}
\sigma_\varepsilon(A) = \{\lambda \in \mathbb{C} : \ \|(A - \lambda
I)^{-1}\| > 1/\varepsilon\}
\end{equation}
or as
\begin{equation}\label{nonstrict}
\Sigma_\varepsilon(A) = \{\lambda \in \mathbb{C} : \ \|(A - \lambda
I)^{-1}\| \ge 1/\varepsilon\} ,
\end{equation}
where $\varepsilon > 0$ and $\|(A - \lambda I)^{-1}\|$ is assumed to
be infinite if $\lambda \in \sigma(A)$ (see, e.g.,
\cite{Bo,BG,BGS,Da,Tre,TE} and \cite{BGj,GHP}). The difference
between $\Sigma_\varepsilon(A)$ and $\sigma_\varepsilon(A)$ is the (closed) level set
\begin{equation}\label{levelset}
\{\lambda \in \mathbb{C} : \ \|(A - \lambda I)^{-1}\| =
1/\varepsilon\}
\end{equation}
and it is natural to ask whether this set may have an open subset,
in which case $\Sigma_\varepsilon(A)$ is strictly larger than the
closure of $\sigma_\varepsilon(A)$. If this happens at a point
$\varepsilon = \varepsilon_0$, then the
$\varepsilon$--pseudospectrum of $A$ jumps as $\varepsilon$ passes
through $\varepsilon_0$.

The question on whether or not the level set \eqref{levelset} may
have an open subset goes back to J.~Globevnik (see \cite{G2}) who
showed that the resolvent norm of a bounded linear operator on a
Banach space cannot be constant on an open set if the underlying
space is complex uniformly convex (see Definition \ref{uconvsmooth}
below). An easy duality argument shows that this remains true if the
dual of the underlying Banach space, rather than the space itself,
is complex uniformly convex. Hence the class of spaces to which
Globevnik's result applies includes Hilbert spaces and $L^p(S,
\Sigma, \mu)$ with $1 \leq p \leq \infty$, where $(S, \Sigma, \mu)$
is an arbitrary measure space (see \cite{S2} or Appendix~A below).

An example of a bounded linear operator on a Banach space for which
the resolvent norm is constant in a neighbourhood of zero was
constructed in \cite{S2} and then modified in \cite{SS} to make the
underlying space separable, reflexive and strictly convex.

According to the above, the resolvent norm of a bounded linear
operator on a Hilbert space cannot be constant on an open set, and
there have been several claims in the literature that the same is
true for a closed densely defined operator on a Hilbert space. A
counterexample to those claims was constructed in \cite{S2}, where
it was shown that there exists a block diagonal closed densely
defined operator on $\ell_2(\mathbb{N})$ with $2\times 2$ blocks,
such that its resolvent norm is constant in a neighbourhood of zero.
It is natural to ask whether this phenomenon can occur for
``non-pathological'' unbounded operators arising in ``real''
applications.

It was shown in \cite{S3} that the answer to this question is
negative for semigroup generators: the resolvent norm of the
infinitesimal generator of a $C_0$ semigroup on a Banach space
cannot be constant on an open set if the underlying space is complex
uniformly convex. This result can also be easily derived with the help of the
Hille-Yosida theorem from a recent result by
S. B\"ogli and P. Siegl (\cite{BS}) about closed operators
acting on a complex uniformly convex Banach space, which says that
if the resolvent norm is constant on an open set, then this constant is the global minimum.
The examples from \cite{S2} and \cite{SS} show
that one cannot drop the requirement of complex uniform convexity in
these results, relax it to complex strict convexity or even replace it with
strict convexity.

Here, we consider another important class of unbounded operators,
namely operators with compact resolvents. In Theorem~\ref{strictlyA}
we show that the resolvent norm of such an operator cannot be
constant on an open set if the underlying Banach space is complex
\textbf{strictly} convex. So, the situation here is slightly different
from what one has for bounded operators. Unlike previous results, Theorem~\ref{strictlyA} is applicable to small perturbations of operators
such as $(Af)(m,n)=(m+in)f(m,n)$ acting in
$l^2(\mathbb{N}\times\mathbb{N})$; see the case $a=0$ of \cite[Theorem~11.1.3]{Da}. This operator has compact
resolvent and the resolvent norm is uniformly bounded away from $0$
for $\lambda\notin \sigma(A)$. In Theorem~\ref{exampA} we show that
the example in \cite{S2} can be modified to produce an operator on a
suitable Banach space $\Xfont$ for which the resolvent is compact and the
resolvent norm is constant in a neighbourhood of zero. Perhaps the most interesting part of this result is that
$\Xfont$ is isometric to
a Hilbert space on a subspace of co-dimension $2$, and we show in Theorem~\ref{exampB} that the same formula as
in the example in \cite{S2} defines a bounded linear operator on $\Xfont$ whose
resolvent norm is constant in a neighbourhood of zero. Theorems~\ref{co1b} and \ref{co1a} show
that similar examples cannot exist in the co-dimension $1$ case.

We use $A$ to denote
a closed densely defined operator with a compact resolvent and $B$ to denote a bounded linear operator or
the infinitesimal generator of a $C_0$ semigroup. We denote by $X$ and $Y$  Banach spaces
satisfying certain convexity hypothesis, while the calligraphic letters
$\Xfont$ and $\Yfont$ are used to denote the spaces $\ell^2(\mathbb{Z})\oplus \mathbb{C}^2$ and $Y\oplus \mathbb{C}$
equipped with suitable norms.
We devote Appendix~A to presenting some known results on convexity properties and absolute norms that are used in the paper.

We conclude this introduction by listing the theorems in the order that they appear below and make brief comments about each one; these comments do not pretend to give full descriptions
of the conditions in the theorems. The symbol $\exists$ denotes that the theorem proves the existence of an operator in some stated
class whose resolvent norm has at least one level set with non-empty interior, while N denotes that no operator in
some stated class possesses a resolvent whose norm has such a level set.
\begin{center}
\begin{tabular}{clcl}
\textbf{Theorem}&\textbf{Banach space}&$\exists$/N&\textbf{Comments}\\
\ref{strictlyA} & $X$  & N&  $A$ has a compact resolvent;\\
\ref{exampA} & $\ell^2(\mathbb{Z})\oplus \mathbb{C}^2$  & $\exists$&  $A$ has a compact resolvent;\\
\ref{exampB} & $\ell^2(\mathbb{Z})\oplus \mathbb{C}^2$  & $\exists$&  $B$ is bounded;\\
\ref{co1} & $Y\oplus_\infty\mathbb{C}$  & $\exists$&  $A$  has a compact resolvent, \\ &&& $B$ is  bounded;\\
\ref{co1b} & $Y\oplus\mathbb{C}$ & N& $A$ has a compact resolvent;\\
\ref{co1a} & $Y\oplus\mathbb{C}$ & N& $B$ generates a $C_0$ semigroup.
\end{tabular}
\end{center}

\section{Main results}

Some of our results depend on the following classical theorem; see  \cite[Theorem~2.16.5]{HP} or \cite[Ch. III, Theorems 5.30 and 6.22]{Kat}.

\begin{theorem}\label{dualresolvthm}
Let $H$ be a closed densely defined operator acting in the Banach space $X$. Then its dual $H^\ast$ has a weak* dense domain in
$X^\ast$ and its graph is weak* closed. Moreover $\Spec(H)=\Spec(H^\ast)$ and
\begin{equation}
\{(H-\lambda I)^{-1}\}^\ast=(H^\ast-\lambda I)^{-1}\label{AA*resthm}
\end{equation}
for all $\lambda\notin\Spec(H)$. In particular
\[
\|(H-\lambda I)^{-1}\|=\|(H^\ast-\lambda I)^{-1}\|
\]
for all $\lambda\notin\Spec(H)$.
\end{theorem}

Our next theorem holds when $X$ is a Hilbert space, and appears to be new even in that case.

\begin{theorem}\label{strictlyA}
Suppose a  Banach space $X$ or its dual $X^*$ is complex strictly convex in the sense of Definition~\ref{uconvsmooth}, and $A : X \to X$ is a closed densely defined operator with
a compact resolvent $R(\lambda) := (A - \lambda I)^{-1}$. Let
$\Omega$ be an open subset of the resolvent set of $A$. If
$\|R(\lambda)\| \le M$ for all $\lambda \in \Omega$, then
$\|R(\lambda)\| < M$ for all $\lambda \in \Omega$.
\end{theorem}
\begin{proof}
The proof is similar to that of \cite[Theorem 2.6]{S2}. One can assume without loss of generality that $\Omega$ is connected. Indeed, it
is sufficient to consider each connected component of $\Omega$ separately.

\emph{Part 1.}  We consider first the case in which $X$ is complex strictly convex.
Suppose that there exists $\lambda_0 \in \Omega$ such that $\|R(\lambda_0)\| = M$. Then  \cite[Theorem 2.1]{S2} or the maximum principle (see, e.g., \cite[Theorem 3.13.1]{HP} or \cite[Ch. III, Sect. 14]{DS}) imply that $\|R(\lambda)\| = M$, $\forall \lambda \in \Omega$.
Shifting the independent variable if necessary, we can assume that
$0 \in \Omega$.

According to \cite[Lemma 1.1]{GV}, there exists $r > 0$ such that
$$
\left\|R(0) + \lambda R^2(0)\right\| = \left\|R(0) + \lambda
R'(0)\right\| \le M, \ \ \ |\lambda| \le r.
$$
Since $\|R(0)\| = M$, there exist $u_n \in X$, $n \in \mathbb{N}$
such that $\|u_n\| = \frac1M$ and $\|R(0)u_n\| \to 1$ as $n \to
\infty$. Since $R(0)$ is compact, one can assume, after going to a
subsequence, that $R(0)u_n$ converges to a vector $x \in X$ and
$\|x\| = 1$. Then $y := r R(0)x \not= 0$ and
\begin{eqnarray*}
\|x + \zeta y\| = \lim_{n \to \infty} \left\|R(0)u_n + \zeta r
R^2(0)u_n\right\| \le
\left\|R(0) + \zeta r R^2(0)\right\|\, \|u_n\| \\
\le M\, \frac1M = 1 , \ \ \ |\zeta| \le 1 .
\end{eqnarray*}
The contradiction implies that there does not exist $\lambda_0\in\Omega$ such that $\| R(\lambda_0)\|=M$.

\emph{Part 2.}  Let us now consider the case where $X^*$  is complex strictly convex. If $X$
is reflexive, then $A^*$ is a closed, densely defined operator (see, e.g., \cite[Theorems 2.11.8
and 2.11.9]{HP} or \cite[Ch. III, Theorem 5.29]{Kat}) with
a compact resolvent, and our claim follows by duality from what has been proved in Part 1 above.
If $X$ is not reflexive, $A^*$ might not be
densely defined, but one can repeat the argument by using Theorem~\ref{dualresolvthm}. This implies that the resolvent $R^*(\lambda)$ of $A^\ast$ is compact and one-one on $X^\ast$, and that its range is weak* dense in $X^\ast$ and independent of $\lambda$. One can now proceed as in Part~1.
\end{proof}

Let us consider the following norm on $l_2(\mathbb{Z})$:
\begin{equation}
\|x\|_* = \max\left\{\|x'\|_2, |x_1|\right\}+ |x_0| , \ \ \ x =
(x_k)_{k \in \mathbb{Z}}, \ x' =  (x_k)_{k \in
\mathbb{Z}\setminus\{0, 1\}} .\label{starnormeq}
\end{equation}
It is easy to see that
\begin{eqnarray*}
&& \|x\|_* \le \|\tilde{x}\|_2 + |x_0| \le \sqrt{2}\, \|x\|_2 ,
\ \ \ \tilde{x} =  (x_k)_{k \in \mathbb{Z}\setminus\{0\}} ,\\
&& \|x\|_* \ge \frac{1}{\sqrt{2}}\, \|\tilde{x}\|_2 + |x_0|  \ge
\frac{1}{\sqrt{2}}\, \|x\|_2 .
\end{eqnarray*}
Hence
\begin{equation}
\frac{1}{\sqrt{2}}\, \|x\|_2 \le \|x\|_* \le \sqrt{2}\, \|x\|_2 , \
\ \ \forall x \in l_2(\mathbb{Z}) .\label{sqrt}
\end{equation}
Moreover $\|x\|_2=\|x\|_*$ if $x_0=x_1=0$. Letting $\Xfont$ denote the
space $l_2(\mathbb{Z})$ equipped with the norm $\|\cdot\|_*$, it
follows that $\Xfont$ is reflexive and separable.

Note that $\Xfont = \mathcal{H}\oplus \mathbb{C}^2$ where
\[
\mathcal{H}=\{ x\in l_2(\mathbb{Z}):x_0=x_1=0\}
\]
is a Hilbert space with the norm induced by $\| \cdot\|_\ast$.

\begin{theorem}\label{exampA}
There exists a closed, densely defined operator $A : \Xfont \to \Xfont$ with a
compact resolvent such that $\|(A - \lambda I)^{-1}\|$ is constant
in a neighbourhood of $\lambda = 0$.
\end{theorem}
\begin{proof}
The proof is similar to that of Theorem 3.1 in \cite{S2}. We suppose
throughout that $\delta=\frac14$ and that $\lambda \in \mathbb{C}$ is arbitrary subject to
\begin{equation}
|\lambda| \leq \delta.\label{exampeq0}
\end{equation}
The same calculations are applicable for any smaller positive value of  $\delta$.

\emph{Part 1.} Let $A$ be the weighted shift
operator defined by
$$
(Ay)_k = \delta^{-|k|} y_{k + 1} , \ \ \ k \in \mathbb{Z} ,
$$
where $\Dom(A)$ is the set of all $y\in \Xfont$ for which $Ay\in \Xfont$. Since $\Dom(A)$ contains
all sequences with finite support, it is dense in
$\Xfont$. It is clear that $A : \Dom(A) \to \Xfont$ is invertible and
\begin{equation}
(A^{-1}x)_k = \beta_k x_{k - 1} , \ \ \ k \in \mathbb{Z}
,\label{exampeq1}
\end{equation}
for all $x\in \Xfont$, where
\begin{equation}
\beta_k = \delta^{|k-1|}, \ \ \ k \in \mathbb{Z}.\label{exampeq2}
\end{equation}
The formula $\lim_{k\to\pm\infty} \beta_k=0$ implies that $A^{-1} :
\Xfont \to \Xfont$ is a compact operator.

\emph{Part 2.}  If one considers $A^{-1}$ as an operator on $l_2(\mathbb{Z})$ equipped with the standard norm,
then it is clear that $\|A^{-1}\| = 1$ (see \eqref{exampeq1}, \eqref{exampeq2}), and hence
\begin{equation}
(A - \lambda I)^{-1} = \sum_{j = 0}^\infty \lambda^j A^{-(j + 1)} , \ \ \ |\lambda| < 1 .
\label{exampeq4}
\end{equation}
Since $\Xfont$ coincides with $l_2(\mathbb{Z})$ as a set and is equipped with a norm equivalent to that of $l_2(\mathbb{Z})$,
we conclude that $A - \lambda I : \Dom(A) \to \Xfont$ is invertible
when $|\lambda| < 1$, and hence when $\lambda$ satisfies (\ref{exampeq0}), and that \eqref{exampeq4} remains valid in this setting.
(Actually, it is not difficult to show that the equality $\|A^{-1}\| = 1$ holds for the operator $A^{-1} : \Xfont \to \Xfont$. Equality \eqref{exampeq8} below, which is the main claim of the theorem, is a considerably stronger statement.)

\emph{Part 3.} Take an arbitrary $x \in \Xfont$ such that $\|x\|_* \le
1$, and note that this implies $| x_k|\le 1-| x_0|$ for all $k\not=
0$. Assume that $\lambda \in \mathbb{C}$ satisfies (\ref{exampeq0}).
Let $y = (A - \lambda I)^{-1}x$. Since
\[
(A^{-(j + 1)}x)_k  = \beta_k\beta_{k-1}\cdots\beta_{k - j}\, x_{k - 1 - j}
\]
for all $k\in \mathbb{Z}$ and all $j \in \mathbb{N}\cup\{0\}$, one has
\begin{eqnarray*}
|y_0|&=& |\beta_0x_{-1}+\lambda\beta_0\beta_{-1}x_{-2}+
\lambda^2\beta_0\beta_{-1}\beta_{-2}x_{-3}+\cdots|\\
&\le  & \delta(1-|x_0| )(1+|\lambda|+|\lambda|^2+\cdots)\\
&=&\delta\frac{1-|x_0|}{1-|\lambda|},
\end{eqnarray*}
and
\begin{eqnarray*}
|y_1|&=& |\beta_1x_{0}+\lambda\beta_1\beta_{0}x_{-1}+
\lambda^2\beta_1\beta_{0}\beta_{-1}x_{-2}+\cdots|\\
&\le  & |x_0|+\delta(1-|x_0| )(|\lambda|+|\lambda|^2+\cdots)\\
&=&|x_0|+\delta|\lambda|\frac{1-|x_0|}{1-|\lambda|}.
\end{eqnarray*}
Combining these bounds yields
\begin{eqnarray}
|y_1|+|y_0|&\le & |x_0|+\delta|\lambda|\frac{1-|x_0|}{1-|\lambda|}
+\delta\frac{1-|x_0|}{1-|\lambda|}\nonumber\\
&= & |x_0|+\delta\frac{1+|\lambda|}{1-|\lambda|}(1-|x_0|)\nonumber\\
&\leq & |x_0|+(1-|x_0|)\nonumber\\
&= &1.\label{exampeq5}
\end{eqnarray}

\emph{Part 4.} We use the fact that $\|x\|_*\le 1$ implies
$|x_k|\leq 1$ for all $k\in\mathbb{Z}$. Then
\begin{eqnarray*}
|y_k|&=& |\beta_kx_{k-1}+\lambda\beta_k\beta_{k-1}x_{k-2}+
\lambda^2\beta_k\beta_{k-1}\beta_{k-2}x_{k-3}+\cdots|\\
&\le  & \delta^{|k-1|}(1+|\lambda|+|\lambda|^2+\cdots)\\
&=&\frac{\delta^{|k-1|}}{1-|\lambda|}\le \frac43\,\delta^{|k-1|}.
\end{eqnarray*}
In particular $|y_0|\leq \frac13$.
Hence
\begin{eqnarray}
\left(\sum_{k \not= 0, 1} |y_k|^2\right)^{1/2} + |y_0| \le
\left( \left( \frac{4}{3} \right)^2 \sum_{k \le -1} \delta^{2(1-k)} +
\left( \frac{4}{3} \right)^2\sum_{k \ge 2} \delta^{2(k - 1)} \right)^{1/2} + \frac13 \nonumber\\
\le \frac43 \left( \frac{\delta^4}{1 - \delta^2} +\frac{\delta^2}{1 -
\delta^2}\right)^{1/2} + \frac13 =
\frac{4\delta}{3} \left( \frac{1+\delta^2}{1 - \delta^2}\right)^{1/2} +\frac13<1.
\hspace*{2em} \label{exampeq6}
\end{eqnarray}

\emph{Part 5.} By combining the bounds (\ref{exampeq5}) and
(\ref{exampeq6}) we obtain $\|y\|_* \le 1$ and hence
\begin{eqnarray}
\|(A - \lambda I)^{-1}\| \le 1 .\label{exampeq7}
\end{eqnarray}
On the other hand, let $z = (A - \lambda I)^{-1} e_0$,  where $e_0
:= (\dots, 0, 0, 1, 0, 0, \dots)$ with 1 at the $0^{\rm th}$ place.
Then  $\|e_0\|_* = 1$,
$$
z_k = \left\{ \begin{array}{ll}
0 , &   k \le 0 , \\
\beta_1 , &   k = 1 , \\
\beta_1\beta_2\cdots \beta_k \lambda^{k - 1} , \ \ &   k \ge 2 ,
\end{array} \right.
$$
and
$$
\left\|(A - \lambda I)^{-1} e_0\right\|_* = \|z\|_* \ge \beta_1 = 1.
$$
By combining this with (\ref{exampeq7}), we finally deduce that
\begin{equation}
 \|(A - \lambda I)^{-1}\|= 1 \label{exampeq8}
\end{equation}
under the condition (\ref{exampeq0}).
\end{proof}

\begin{theorem}\label{exampB}
Let $\Xfont$ denote the space $l_2(\mathbb{Z})$ equipped with the norm $\|\cdot\|_*$ defined in \eqref{starnormeq}.
Then there exists an invertible bounded linear operator $B : \Xfont \to \Xfont$ such that
$\|(B - \lambda I)^{-1}\|$ is constant in a neighbourhood of $\lambda = 0$.
\end{theorem}
\begin{proof}
\emph{Part 1.}
Let $M > 3$ and let $B$ be the weighted shift operator
$$
(Bx)_k = \alpha_k x_{k + 1} , \ \ \ k \in \mathbb{Z} ,
$$
where
$$
\alpha_k = \left\{ \begin{array}{cl}
\frac{1}{M}\,  , &   k = 0 , \\
1 , &   k \not= 0 .
\end{array} \right.
$$
It is clear that $B$ is invertible on $\Xfont$ and that
$$
(B^{-1}y)_k = \beta_k y_{k - 1} , \ \ \ k \in \mathbb{Z} ,
$$
where
$$
\beta_k = \frac{1}{\alpha_{k - 1}} = \left\{ \begin{array}{cl}
M  , &   k = 1 , \\
1 , &   k \not= 1 .
\end{array} \right.
$$
As in Part 2 of the proof of Theorem \ref{exampA}, one can
 consider $B^{-1}$ as an operator on $l_2(\mathbb{Z})$ equipped with the standard norm.
Then it is clear that $\|B^{-1}\| = M$, and hence
\begin{equation}
(B - \lambda I)^{-1} = \sum_{j = 0}^\infty \lambda^j B^{-(j + 1)} , \ \ \ |\lambda| <  \frac{1}{M}\, .
\label{exampeq4B}
\end{equation}
Again, since $\Xfont$ coincides with $l_2(\mathbb{Z})$ as a set and is equipped with a norm equivalent to that of $l_2(\mathbb{Z})$,
we conclude that $B - \lambda I : \Xfont \to \Xfont$ is invertible
when $|\lambda| < \frac{1}{M}$, and that \eqref{exampeq4B} remains valid in this setting.
(It is not difficult to show that the equality $\|B^{-1}\| = \frac{1}{M}$ holds for the operator $B^{-1} : \Xfont \to \Xfont$.)

\emph{Part 2.}
Take an arbitrary $x \in \Xfont$ such that $\|x\|_* \le 1$ and an
arbitrary $\lambda \in \mathbb{C}$ such that $|\lambda| < \frac{1}{M}$.
Let $y = (B - \lambda I)^{-1}x$. Since  $(B^{-(j + 1)}x)_k = \beta_k\cdots\beta_{k - j}\, x_{k - 1 - j}$,
$j \in \mathbb{N}\cup\{0\}$, we get
\begin{eqnarray}
y_k &= &\left\{ \begin{array}{ll}
\sum_{j = 0}^\infty \lambda^j x_{k - 1 - j} , &   k \le 0 , \\
Mx_0 +  M \sum_{j = 1}^\infty \lambda^j x_{-j} , &   k = 1 , \\
\sum_{j = 0}^{k - 2} \lambda^j x_{k - 1 - j}  +
M \lambda^{k - 1} x_0 +
M \sum_{j = k}^\infty \lambda^j x_{k - 1 - j} , &   k \ge 2 .
\end{array} \right.\label{exampeq5B}
\end{eqnarray}
In this part we obtain upper bounds on $y$ by using the decomposition $y' = v' + w'$, where $y'$ is defined as in \eqref{starnormeq} and
\begin{eqnarray*}
&& v_k = \sum_{j = 0}^\infty \lambda^j x_{k - 1 - j} = \sum_{l= 1}^\infty \lambda^{l - 1} x_{k - l} , \ \ \  k \in \mathbb{Z}\setminus\{0, 1\}  ,\\
&& w_k = \left\{ \begin{array}{ll}
0 , &   k < 0 , \\
(M - 1) \sum_{j = k - 1}^\infty \lambda^j x_{k - 1 - j} , &   k \ge 2 .
\end{array} \right.
\end{eqnarray*}
Using the notation
$$
\nu = (\nu_l)_{l \in \mathbb{Z}}\, , \ \ \  \nu_l = \left\{ \begin{array}{ll}
0 , &   l < 1 , \\
 \lambda^{l - 1} , &   l \ge 1 ,
\end{array} \right.
$$
we get
\begin{eqnarray*}
\|v'\|_2 \le \|\nu\ast x\|_2 \le \|\nu\|_1 \|x\|_2 = \frac{1}{1 - |\lambda|}\,  \|x\|_2 \le \frac{\sqrt{2}}{1 - |\lambda|}\,  \|x\|_*
= \frac{\sqrt{2}}{1 - |\lambda|}
\end{eqnarray*}
(see \eqref{sqrt}).
Further,
\begin{eqnarray*}
|w_k| &\le&  (M - 1) \left(\sum_{j = k - 1}^\infty |\lambda|^{2j}\right)^{1/2}  \|x\|_2 = \frac{(M - 1) |\lambda|^{k - 1}}{(1 - |\lambda|^2)^{1/2}}\,
\|x\|_2 \\
&\le& \frac{\sqrt{2}\, (M - 1) |\lambda|^{k - 1}}{(1 - |\lambda|^2)^{1/2}}\, , \ \ \ k \ge 2, \\
\|w'\|_2 &\le& \frac{\sqrt{2}\, (M - 1)}{(1 - |\lambda|^2)^{1/2}} \left(\sum_{k = 2}^\infty |\lambda|^{2(k - 1)}\right)^{1/2} =
\frac{\sqrt{2}\, (M - 1) |\lambda|}{(1 - |\lambda|^2)}\, .
\end{eqnarray*}
Hence,
\begin{eqnarray}
\|y'\|_2 &\le& \|v'\|_2 + \|w'\|_2 \le \frac{\sqrt{2}}{1 - |\lambda|} + \frac{\sqrt{2}\, (M - 1) |\lambda|}{(1 - |\lambda|^2)} \nonumber\\
&=& \frac{\sqrt{2}\, (M|\lambda| + 1)}{(1 - |\lambda|^2)}\, .
\label{exampeq4c}
\end{eqnarray}

Since $\|x\|_* \le 1$ implies $|x_k| \le 1 - |x_0|$, $k \not= 0$,
we may use (\ref{exampeq5B}) directly to obtain
\begin{eqnarray}
&& |y_0| \le (1 - |x_0|) \sum_{j = 0}^\infty |\lambda|^j = \frac{1 - |x_0|}{1 - |\lambda|}\, ,\label{exampeq4d} \\
&& |y_1| \le M |x_0| + M (1 - |x_0|) \sum_{j = 1}^\infty |\lambda|^j = M |x_0| + M |\lambda|\, \frac{1 - |x_0|}{1 - |\lambda|}\, .\label{exampeq4e}
\end{eqnarray}

\emph{Part 3.}
Suppose additionally that $|\lambda| < \frac13 - \frac{1}{M}$. By using (\ref{exampeq4c}),(\ref{exampeq4d}) and (\ref{exampeq4e}) we obtain
\begin{eqnarray*}
\|y'\|_2 +  |y_0| &\le& \frac{\sqrt{2}\, (M|\lambda| + 1)}{(1 - |\lambda|^2)} + \frac{1 - |x_0|}{1 - |\lambda|} <
\frac{2 (M|\lambda| + 1)}{1 - |\lambda|} + \frac{1}{1 - |\lambda|} \\
&=& \frac{2 (M|\lambda| + 1) + 1}{1 - |\lambda|} < \frac{2\, \frac{M}3 + 1}{\frac23 + \frac1M} = \frac{2M + 3}{2M + 3}\, M = M
\end{eqnarray*}
and
\begin{eqnarray*}
|y_1| + |y_0| \le M |x_0| + (1 - |x_0|)\, \frac{1 + M |\lambda|}{1 - |\lambda|}
\le M|x_0| + \frac{\frac{M}{3}}{\frac23 + \frac{1}{M}}\, (1 - |x_0|) \\ =
M|x_0| + \frac{M}{2M + 3}\, M (1 - |x_0|) \le M.
\end{eqnarray*}
Therefore $\|y\|_* \le M$ and
\begin{eqnarray}
\|(B - \lambda I)^{-1}\| \le M , \ \ \ |\lambda| < \min\left\{\frac{1}{M}\,, \ \frac13 - \frac{1}{M}\right\} .\label{exampeq6a}
\end{eqnarray}

\emph{Part 4.} The proof is completed by combining (\ref{exampeq6a}) with a corresponding, but simpler, lower bound.
If $z = (B - \lambda I)^{-1} e_0$, where $e_0$ is the same as in Part 5 of the proof of Theorem \ref{exampA}, then
$$
z_k = \left\{ \begin{array}{ll}
0 , &   k \le 0 , \\
M , &   k = 1 , \\
M \lambda^{k - 1} , &   k \ge 2 ,
\end{array} \right.
$$
and $\|z\|_* = \left\|(B - \lambda I)^{-1} e_0\right\|_* \ge M$,
$|\lambda| < \frac{1}{M}$. Therefore
$$
\|(B - \lambda I)^{-1}\| \ge M , \ \ \ |\lambda| < \min\left\{\frac{1}{M}\,, \ \frac13 - \frac{1}{M}\right\} .
$$
\end{proof}

\section {Co-dimension one}

We have shown above that there exist a bounded operator and a closed densely defined operator with a
compact resolvent on $\Xfont = \left(l_2(\mathbb{Z})\oplus_\infty \mathbb{C}\right)\oplus_1\mathbb{C}$ whose
resolvent norms are constant in a neighbourhood of $0$. The norm in $X$ coincides with the $l_2$ norm
on a subspace of co-dimension two, and it is natural to ask whether similar examples exist in co-dimension one.

\begin{theorem}\label{co1}
There exist a separable, complex uniformly convex Banach space $Y$, a bounded linear operator
$B$ and a closed densely defined operator $A$ with a
compact resolvent on $\Yfont := Y\oplus_\infty \mathbb{C}$ whose
resolvent norms are constant in a neighbourhood of $0$.
\end{theorem}
\begin{proof}
Let $c_0(\mathbb{Z})$ denote as usual the subspace of $l_\infty(\mathbb{Z})$ consisting of all elements
$(x_k)_{k \in \mathbb{Z}}$ such that $\lim_{|k| \to \infty} x_k = 0$, and let $B_1$ be the operator obtained from the one
in \cite[Theorem 3.1]{S2} if one replaces $l_\infty(\mathbb{Z})$ by $Y_1:=c_0(\mathbb{Z})$. The proof of that theorem carries over to
$B_1$ without change. Therefore the resolvent norm of $B_1 : Y_1\oplus_1\mathbb{C} \to
Y_1\oplus_1\mathbb{C}$ is constant in a neighbourhood of $0$. The dual space $Y:= l_1(\mathbb{Z})$ of
$c_0(\mathbb{Z})$ is complex uniformly convex (see \cite{G1}), and the resolvent norm of the adjoint operator
$B := B_1^* : Y\oplus_\infty \mathbb{C} \to Y\oplus_\infty \mathbb{C}$ is  constant in a neighbourhood of $0$.

Similarly, one can define an operator $A_1$ by the same formula as in the proof of Theorem~\ref{exampA}, but
replacing $l_2(\mathbb{Z})$ by $Y_1:=c_0(\mathbb{Z})$ there. Then an argument similar to, but easier than, the proof of that theorem shows that $A_1 : Y_1\oplus_1\mathbb{C} \to
Y_1\oplus_1\mathbb{C}$ is a closed densely defined operator with a compact resolvent
whose norm is constant in a neighbourhood of $0$. It is follows that $A := A_1^* : Y\oplus_\infty \mathbb{C} \to Y\oplus_\infty \mathbb{C}$
has the desired properties.
\end{proof}

Our next task is to show that one cannot take $Y$ to be a Hilbert space in the above theorem: the resolvent norm of
a bounded linear operator or of a closed densely defined operator with a compact resolvent on
$l_2\oplus_p \mathbb{C}$, $1 \le p \le \infty$ cannot be constant on an open set.
In fact, we prove a more general result which uses absolute norms on $\mathbb{C}^2$. Readers not familiar with such norms and with the definitions of $\Psi$ and $\oplus_\psi$ where $\psi\in\Psi$ should refer to Appendix~A.

\begin{theorem}\label{co1b}
Let $\Omega$ be an open subset of $\mathbb{C}$
and let $Y$ be a complex strictly convex Banach space with a complex strictly convex dual $Y^*$.
Given $\psi \in \Psi$, let $\|\cdot\|$ be the ($\psi$-dependent) norm on the Banach space
$\Yfont := Y\oplus_\psi \mathbb{C}$. Suppose
$A : \Yfont \to \Yfont$ is a closed densely defined operator with a compact resolvent $(A - \lambda I)^{-1}$
defined for all $\lambda \in \Omega$. If $\|(A - \lambda I)^{-1}\| \le M$
for all $\lambda \in \Omega$, then $\|(A - \lambda I)^{-1}\| < M$ for all $\lambda \in \Omega$.
\end{theorem}
\begin{proof}
One can assume as in the proof of Theorem \ref{strictlyA} that $\Omega$ is connected.

\emph{Part 1.}  If $\psi$ satisfies \eqref{cc}, then $\Yfont$ is complex strictly convex and our claim follows
from Theorem~\ref{strictlyA}.

\emph{Part 2.}  Suppose there exist $t_0 \in (0, 1/2]$ and $t_1 \in [1/2, 1)$ such that $\psi(t_0) = 1 - t_0$ and $\psi(t_1) = t_1$.
Then $\psi^*$ satisfies \eqref{cc} (see \eqref{psidualeq}--\eqref{t}),  and $Y^*\oplus_{\psi^*} \mathbb{C}$ is complex strictly convex.
Since $\Yfont^* = Y^*\oplus_{\psi^*} \mathbb{C}$ (see \cite{LF}), our claim again follows from Theorem~\ref{strictlyA}.

\emph{Part 3.} Suppose $\psi(t) > t$ for all $t \in [1/2, 1)$ and $\psi(t) = 1 - t$ for sufficiently small $t > 0$. Then there exists
$t_0 \in (0, 1/2)$ such that $\psi(t) = 1 - t$ for all $t \in [0, t_0]$ and $\psi(t) > 1 - t$ for all $t \in (t_0, 1/2]$.

Suppose there exists
$\lambda_0 \in \Omega$ such that $\|R(\lambda_0)\| = M$, where
$R(\lambda) := (A - \lambda I)^{-1}$. Then, exactly as in the proof of Theorem~\ref{strictlyA}, one can assume that
$0 \in \Omega$ and derive from \cite[Lemma 1.1]{GV} the existence of  $r > 0$ such that
$\|R(0) +  \lambda R'(0)\| \le M$ and $\|R(0) + \lambda R''(0)\| \le M$
when $|\lambda| \le r$. Hence
\begin{equation}\label{res23b}
\|R(0) + \lambda R^2(0)\| \le M , \ \ \|R(0) + \lambda R^3(0)\| \le M , \ \ \ |\lambda| \le r .
\end{equation}

Suppose $M_0 : = \|P_0R(0)\| < M$ (see \eqref{P01}). By continuity, there exists $\delta_2 > 0$ such that
\begin{equation}\label{delta1A}
\psi(t) \ge \frac{M}{M_0} (1 - t) \ \ \Longrightarrow \ \ t \ge t_0 + \delta_2 .
\end{equation}
There clearly exists $\psi_1 \in \Psi$ that satisfies \eqref{cc} and the following condition
$$
 \psi_1(t) = \psi(t) ,  \ \ \ \forall t \in \left[t_0 + \frac{\delta_2}2\, , 1\right]  .
$$
Let $\|\cdot\|'$ denote the norm on $\Yfont_1 := Y\oplus_{\psi_1} \mathbb{C}$. Note that $\Yfont$ and $\Yfont_1$
coincide as vector spaces but are equipped with different norms $\|\cdot\|$ and $\|\cdot\|'$, which are
equivalent to each other (see \eqref{absnorms}).

Since $\|R(0)\| = M$, there exist $u_n \in \Yfont$, $n \in \mathbb{N}$
such that $\|u_n\| = \frac1M$ and $\|R(0)u_n\| \to 1$ as $n \to
\infty$. Since $R(0)$ is compact, one can assume, after going to a
subsequence, that $R(0)u_n$ converges to a vector $x \in \Yfont$ and
$\|x\| = 1$.

Denoting for brevity $z = \|P_0 x\|$,
$v = |P_1 x|$, we get
$$
z = \lim_{n \to \infty} \|P_0 R(0)u_n\| \le \|P_0R(0)\| \|u_n\| =   M_0/M
$$
and
\begin{eqnarray*}
&& (z + v)\, \psi\left(\frac{v}{z + v}\right) = \|x\| = 1 \ \ \Longrightarrow \\
&& \psi\left(\frac{v}{z + v}\right) = \frac{1}{z + v} = \frac{1}z \left(1 - \frac{v}{z + v}\right)  \ge
 \frac{M}{M_0} \left(1 - \frac{v}{z + v}\right)  \ \ \Longrightarrow \\
&&  \frac{v}{z + v} \ge t_0 + \delta_2  \ \ \Longrightarrow \ \ \psi_1\left(\frac{v}{z + v}\right) = \psi\left(\frac{v}{z + v}\right) \ \
\Longrightarrow \\
&&  \|x\|' = \|x\| = 1
\end{eqnarray*}
(see \eqref{normpsi} and \eqref{delta1A}).
Also, by continuity, there exists $r_0 \in (0, r]$ such that
\begin{eqnarray*}
\frac{|P_1 x + \lambda P_1 R(0)x|}{\|P_0 x + \lambda P_0 R(0)x\| + |P_1 x + \lambda P_1 R(0)x|}
\ge \frac{v}{z + v} - \frac{\delta_2}2 \\
\ge t_0 + \frac{\delta_2}2, \ \ \ |\lambda| \le r_0 ,
\end{eqnarray*}
and hence $\|x + \lambda R(0)x\|' = \|x + \lambda R(0)x\|$ when $ |\lambda| \le r_0$.

Since $\psi_1$ satisfies \eqref{cc}, $\Yfont_1$ is complex strictly convex. Let $y := r_0 R(0)x$. Then  $y \not = 0$.
On the other hand, the first inequality in \eqref{res23b} implies
\begin{eqnarray*}
&& \|x + \zeta y\|' = \|x + \zeta r_0 R(0)x\|' = \|x + \zeta r_0 R(0)x\| \\
&& = \lim_{n \to \infty} \|R(0)u_n + \zeta r_0 R^2(0)u_n\|
\le \|R(0) + \zeta r_0 R^2(0)\|\, \|u_n\| \\
&& \le M\, \frac1M = 1 , \ \ \ |\zeta| \le 1 .
\end{eqnarray*}
Since $\|x\|' = \|x\| =  1$, we get a contradiction with the complex strict convexity of $\Yfont_1$.
Hence $\|P_0R(0)\| < M$ cannot hold.

\emph{Part 4.} Since $\|P_0R(0)\| = M$, we can prove as in Part 3  that there exist $u_n \in \Yfont$, $n \in \mathbb{N}$
such that $\|u_n\| = \frac1M$ and $R(0)u_n$ converges to a vector $x \in \Yfont$ with
$\|P_0 x\| = 1$. Suppose $P_0R(0) x = 0$ and $P_0R^2(0) x = 0$. Then
$R(0) x, R^2(0) x \in \mathbb{C}$ and there exist $\mu, \eta \in \mathbb{C}$ such that
$|\mu| + |\eta| = 1$ and $\mu R(0) x + \eta R^2(0) x = 0$. Further,
\begin{eqnarray*}
&& R(0) (\mu x + \eta R(0) x) = 0 \ \ \Longrightarrow \ \ \mu x + \eta R(0) x = 0  \ \ \Longrightarrow \\
&& \mu P_0 x = 0  \ \ \Longrightarrow \ \ \mu = 0 \ \ \Longrightarrow \ \ R^2(0) x = 0 \ \ \Longrightarrow \ \
x = 0.
\end{eqnarray*}
This contradiction shows that at least one of $P_0R(0) x$ and $P_0R^2(0) x$ is nonzero.

\emph{Part 5.} Suppose $P_0R^2(0) x \not= 0$ and let $y_0 := r P_0R^2(0) x$. Then the second inequality in
\eqref{res23b} implies
\begin{eqnarray*}
&& \|P_0 x + \zeta y_0\| = \lim_{n \to \infty} \|P_0 R(0)u_n + \zeta r P_0 R^3(0)u_n\| \\
&& \le \|R(0) + \zeta r_0  R^3(0)\|\, \|u_n\|
 \le M\, \frac1M = 1 , \ \ \ |\zeta| \le 1 .
\end{eqnarray*}
The complex strict convexity of $Y$ implies that $y_0 = 0$.  This contradiction shows that $P_0R^2(0) x \not= 0$
 cannot hold.

\emph{Part 6.} Similarly, one shows that $P_0R(0) x \not= 0$ cannot hold either. Since this exhausts our list of
possibilities, we conclude that there cannot exist $\lambda_0 \in \Omega$ such that $\|R(\lambda_0)\| = M$. This proves our claim
in the case of $\psi$ satisfying the conditions stated at the beginning of Part 3 above.

\emph{Part 7.}  Finally, suppose $\psi(t) > 1 - t$ for all $t \in (0, 1/2]$ and $\psi(t) = t$ for $t$ sufficiently close to $1$. Then
$\psi^*(t) > t$ for all $t \in [1/2, 1)$ (see \eqref{t}). Hence $\psi^*$ satisfies either the conditions in Part 1 or those in Part 3, and
our claim follows by duality from what has already been proved (cf. Part 2 above and Part 2 of the proof of Theorem \ref{strictlyA}).

\end{proof}

\begin{theorem}\label{co1a}
Let $\Omega$ be an open subset of $\mathbb{C}$, $\psi \in \Psi$,
$Y$ be a complex uniformly convex Banach space with a complex uniformly convex dual $Y^*$, and
let $\Yfont := Y\oplus_\psi \mathbb{C}$. Suppose
$B$ is the infinitesimal generator of a $C_0$ semigroup on $\Yfont$
and $B - \lambda I$ is invertible for all
$\lambda \in \Omega$. If $\|(B - \lambda I)^{-1}\| \le M$
for all $\lambda \in \Omega$, then $\|(B - \lambda I)^{-1}\| < M$ for all $\lambda \in \Omega$.
\end{theorem}
\begin{proof}
The proof follows the same lines as  that of Theorem \ref{co1b} but it is somewhat more technical.
We can assume as above that
$\Omega$ is connected.

\emph{Part 1.}  If $\psi$ satisfies \eqref{cc}, then $\Yfont$ is complex uniformly convex and our claim follows
from \cite{S3}.

\emph{Part 2.}  Suppose there exist $t_0 \in (0, 1/2]$ and $t_1 \in [1/2, 1)$ such that $\psi(t_0) = 1 - t_0$ and $\psi(t_1) = t_1$.
Then $\psi^*$ satisfies \eqref{cc}(see \eqref{psidualeq}--\eqref{t}), and $\Yfont^* = Y^*\oplus_{\psi^*} \mathbb{C}$ (see \cite{LF})
is complex uniformly convex. If $\Yfont$ is reflexive, $B^*$ is the infinitesimal generator of a $C_0$ semigroup on $\Yfont^*$
(see, e.g., \cite[Corollary 3.3.9]{{ABHN}}), and our claim follows by duality from the main result in \cite{S3} applied to $B^*$.
If $\Yfont$ is not reflexive, one can use the result in Remark~\ref{semirem} instead of the latter (see also Theorem
\ref{dualresolvthm}).

\emph{Part 3.} Suppose $\psi(t) > t$ for all $t \in [1/2, 1)$ and $\psi(t) = 1 - t$ for sufficiently small $t > 0$. Then there exists
$t_0 \in (0, 1/2)$ such that $\psi(t) = 1 - t$ for all $t \in [0, t_0]$ and $\psi(t) > 1 - t$ for all $t \in (t_0, 1/2]$.

Suppose there exists
$\lambda_0 \in \Omega$ such that $\|R(\lambda_0)\| = M$, where
$R(\lambda) := (B - \lambda I)^{-1}$. Then, exactly as in Part 3 of the proof of Theorem \ref{co1b}, one
arrives at the same estimates as in \eqref{res23b}:
\begin{equation}\label{res23}
\|R(0) + \lambda R^2(0)\| \le M , \ \ \|R(0) + \lambda R^3(0)\| \le M , \ \ \ |\lambda| \le r .
\end{equation}

Suppose $M_0 : = \|P_0R(0)\| < M$. Take $\delta_1 > 0$ such that
$$
\varrho :=  \frac{M(1 - \delta_1)}{M_0} > 1 .
$$
By continuity, there exists $\delta_2 > 0$ such that
\begin{equation}\label{delta1B}
\psi(t) \ge \varrho (1 - t) \ \ \Longrightarrow \ \ t \ge t_0 + \delta_2 .
\end{equation}
There exists $\psi_1 \in \Psi$ that satisfies \eqref{cc} and the condition
$$
 \psi_1(t) = \psi(t) ,  \ \ \ \forall t \in \left[t_0 + \frac{\delta_2}2\, , 1\right]  .
$$
Let $\|\cdot\|'$ denote the norm on $\Yfont_1 := Y\oplus_{\psi_1} \mathbb{C}$.

For any  $\delta \in (0, \delta_1]$, there exists $u \in \Yfont$ such that
$\|u\| = 1/M$ and $\|R(0)u\| > 1 - \delta$. Denoting for brevity $z = \|P_0 R(0) u\|$,
$v = |P_1 R(0) u|$, we get $z \le M_0/M$ and
\begin{eqnarray*}
&& (z + v)\, \psi\left(\frac{v}{z + v}\right) > 1 - \delta \ \ \Longrightarrow \\
&& \psi\left(\frac{v}{z + v}\right) > \frac{1 - \delta}{z + v} = \frac{1 - \delta}z \left(1 - \frac{v}{z + v}\right) \\
&& \ge \frac{M (1- \delta)}{M_0} \left(1 - \frac{v}{z + v}\right) \ge \varrho \left(1 - \frac{v}{z + v}\right) \ \ \Longrightarrow \\
&& \frac{v}{z + v} \ge t_0 + \delta_2  \ \ \Longrightarrow \ \ \psi\left(\frac{v}{z + v}\right) = \psi_1\left(\frac{v}{z + v}\right) \\
&& \Longrightarrow \ \ \|R(0)u\| = \|R(0)u\|' .
\end{eqnarray*}
Also, by continuity, there exists $r_0 \in (0, r]$ such that
\begin{eqnarray*}
\frac{|P_1 R(0)u + \lambda P_1 R^2(0)u|}{\|P_0 R(0)u + \lambda P_0 R^2(0)u\| + |P_1 R(0)u + \lambda P_1 R^2(0)u|}
\ge \frac{v}{z + v} - \frac{\delta_2}2 \\
\ge t_0 + \frac{\delta_2}2, \ \ \ |\lambda| \le r_0 ,
\end{eqnarray*}
and hence $\|R(0)u + \lambda R^2(0)u\|' = \|R(0)u + \lambda R^2(0)u\|$.

Since $\psi_1$ satisfies \eqref{cc}, $\Yfont_1$ is complex uniformly convex. Take an arbitrary $\tau > 0$ and consider
$\delta$ corresponding to $\varepsilon := r_0\tau/2$ in the
definition of complex uniform convexity. Decreasing $\delta$ if necessary, we can assume that $\delta \le \max\{1/2, \delta_1\}$.
Let $x := R(0)u$ and $y := r_0 R^2(0)u$. Then the first inequality in \eqref{res23} implies
\begin{eqnarray*}
\|x + \zeta y\|' = \|R(0)u + \zeta r_0 R^2(0)u\|' = \|R(0)u + \zeta r_0 R^2(0)u\| \\
\le \|R(0) + \zeta r_0 R^2(0)\|\, \|u\| \le
M\, \frac1M = 1 , \ \ \ |\zeta| \le 1 .
\end{eqnarray*}
Since $\|x\|' = \|x\| >  1 - \delta$, the complex uniform convexity of $\Yfont_1$ implies that
$\|y\|' < \varepsilon$. Now it follows from \eqref{absnorms} that  $\|R^2(0)u\| \le 2 \|R^2(0)u\|' < \tau$, i.e $\|B^{-2} u\| < \tau$.

Applying \eqref{21} with $w :=  B^{-2} u \in \mbox{\rm Dom}(B^2)$ and taking into account that $\|w\| < \tau$ and
$$
\|Bw\| = \|B^{-1} u\| = \|R(0)u\| >  1 - \delta \ge 1/2\, , \ \ \
\|B^2w\| = \|u\| =\frac{1} M ,
$$
we obtain
$$
\frac14 \le C \left(\frac1M + \tau\right)\tau ,
$$
where $\tau > 0$ can be taken arbitrarily small. This contradiction shows that $\|P_0R(0)\| < M$ cannot hold, i.e. that
$\|P_0R(0)\| = M$.

\emph{Part 4.} Since $\|P_0R(0)\| = M$, for any $\delta \in (0, 1/2]$ there exists $u \in \Yfont$ such that
$\|u\| = 1/M$ and $\|P_0 R(0)u\| > 1 - \delta$. Suppose $\theta\left(R^2(0) u\right) < \rho$ and
$\theta\left(R^3(0) u\right) < \rho$ (see \eqref{theta}), where $\rho > 0$ is a sufficiently small number to be chosen later.
Since $P_1R^2(0) u, P_1R^3(0) u \in \mathbb{C}$, there exist $\mu, \eta \in \mathbb{C}$ such that
$|\mu| + |\eta| = 1$ and $\mu P_1R^2(0) u + \eta P_1R^3(0) u = 0$. Then \eqref{th2} implies
\begin{eqnarray*}
&&\|\mu R^2(0) u + \eta R^3(0) u\| = \|\mu P_0R^2(0) u + \eta P_0R^3(0) u\| \\
&& \le 2\rho (|\mu| \|R^2(0) u\| + |\eta| \|R^3(0) u\|)\\
&&  \le 2\rho (|\mu| M + |\eta| M^2) \le
2M(M + 1) \rho .
\end{eqnarray*}
Applying \eqref{21} with $w :=  \mu R^2(0) u + \eta R^3(0) u \in \mbox{\rm Dom}(B^2)$, we get
\begin{eqnarray*}
\|\mu R(0) u + \eta R^2(0) u\|^2 \le C\left(\|\mu u + \eta R(0) u\| + 2M(M + 1) \rho\right) 2M(M + 1) \rho  \\
\le C\left(\frac1M + 1 + 2M(M + 1) \rho\right) 2M(M + 1) \rho .
\end{eqnarray*}
Let $\mathcal{M}_1(\rho)$ denote the square root of the right-hand side of the last inequality. Then
\begin{eqnarray*}
\|\mu P_0 R(0) u + \eta P_0 R^2(0) u\| \le \|\mu R(0) u + \eta R^2(0) u\| \\
\le \mathcal{M}_1(\rho)
= O\left(\sqrt{\rho}\right) \ \mbox{ as } \rho \to 0 .
\end{eqnarray*}
Hence
\begin{eqnarray*}
\frac{|\mu|}2 \le  |\mu| (1 - \delta) \le \|\mu P_0 R(0) u\| \le \mathcal{M}_1(\rho)  +
\|P_0 R^2(0) u\| \\
\le \mathcal{M}_1(\rho) + 2\rho \|R^2(0) u\| \le \mathcal{M}_1(\rho) + 2\rho M
=: \mathcal{M}_2(\rho) \\
= O\left(\sqrt{\rho}\right) \ \mbox{ as } \rho \to 0
\end{eqnarray*}
(see \eqref{th2}), and $|\eta| = 1 - |\mu| \ge 1 - 2\mathcal{M}_2(\rho)$. Hence
\begin{eqnarray*}
\left(1 - 2\mathcal{M}_2(\rho)\right) \|R^2(0) u\| \le \|\eta R^2(0) u\| \le
\mathcal{M}_1(\rho) + \|\mu R(0) u\| \\
\le \mathcal{M}_1(\rho) + 2 \mathcal{M}_2(\rho)
\end{eqnarray*}
and
$$
\|R^2(0) u\| \le \frac{\mathcal{M}_1(\rho) + 2 \mathcal{M}_2(\rho)}{1 - 2\mathcal{M}_2(\rho)} =:
\mathcal{M}_3(\rho)
= O\left(\sqrt{\rho}\right) \ \mbox{ as } \rho \to 0 .
$$
Applying \eqref{21} with $w :=  R^2(0) u  \in \mbox{\rm Dom}(B^2)$, we get
$$
\frac14 \le C\left(\frac1M + \mathcal{M}_3(\rho)\right) \mathcal{M}_3(\rho) .
$$
There exists $\rho_0$ depending only on $M$ and on $C$ in \eqref{21} such that the above inequality fails for all  $\rho \in (0, \rho_0]$.
This contradiction shows that
at least one of the inequalities $\theta\left(R^2(0) u\right) \ge \rho_0$ and
$\theta\left(R^3(0) u\right) \ge \rho_0$ has to hold.

\emph{Part 5.} Suppose $\theta\left(R^3(0) u\right) \ge \rho_0$, Here, $u$ is such that $\|u\| = 1/M$ and $\|P_0 R(0)u\| > 1 - \delta$
with a sufficiently small $\delta \in (0, 1/2]$. Take an arbitrary $\tau > 0$ and consider
$\delta$ corresponding to $\varepsilon := r\tau$ in the definition of complex uniform convexity applied to the space $Y$.
Let $x := P_0 R(0)u$ and $y := r P_0 R^3(0)u$. Then the second inequality in \eqref{res23} implies
\begin{eqnarray*}
\|x + \zeta y\| = \|P_0 R(0)u + \zeta r P_0 R^3(0)u\| = \|P_0 (R(0)u + \zeta r R^3(0)u)\| \\
\le \|R(0)u + \zeta r R^3(0)u\|
\le \|R(0) + \zeta r R^3(0)\|\, \|u\| \\
\le M\, \frac1M = 1 , \ \ \ |\zeta| \le 1 .
\end{eqnarray*}
Since $\|x\| >  1 - \delta$, the complex uniform convexity of $Y$ implies that
$\|y\| < \varepsilon$. Hence $\|P_0 R^3(0)u\| < \tau$. Since $\theta\left(R^3(0) u\right) \ge \rho_0$, it follows from
\eqref{th1} that $\|R^3(0)u\| < \tau/\rho_0$. Applying Theorem~\ref{Rota} with $k = 2$, $n= 3$, and $w :=  R^3(0) u  \in \mbox{\rm Dom}(B^3)$, we obtain
$$
\frac18 \le L_{3, 2} \left(\frac1M + \frac{\tau}{\rho_0}\right)^2 \frac{\tau}{\rho_0}\, ,
$$
where $\tau > 0$ can be taken arbitrarily small. This contradiction shows that $\theta\left(R^3(0) u\right) \ge \rho_0$ cannot hold.

\emph{Part 6.} Similarly, one shows that $\theta\left(R^2(0) u\right) \ge \rho_0$ cannot hold either. Since this exhausts our list of
possibilities, we conclude that there cannot exist $\lambda_0 \in \Omega$ such that $\|R(\lambda_0)\| = M$. This proves our claim
in the case of $\psi$ satisfying the conditions stated at the beginning of Part 3 above.

\emph{Part 7.}  Finally, suppose $\psi(t) > 1 - t$ for all $t \in (0, 1/2]$ and $\psi(t) = t$ for $t$ sufficiently close to $1$. Then
$\psi^*(t) > t$ for all $t \in [1/2, 1)$ (see \eqref{t}). Hence $\psi^*$ satisfies either the conditions in Part 1 or those in Part 3, and
one can repeat the above arguments using \eqref{Dkn*} instead of \eqref{21} and \eqref{Dkn} (cf. Part 2).

\end{proof}

\appendix
\section{Auxiliary results on geometry of Banach spaces}

In this first appendix we summarize some well known concepts and theorems that are used in the paper.

\begin{definition}\label{uconvsmooth}
A Banach space $Y$ is called \\
(i) complex strictly convex (strictly convex) if
\begin{eqnarray*}
x, y \in Y, \ \|x\| = 1 \ \text{ and } \ \|x + \zeta y\| \le 1, \ \ \forall \zeta \in \mathbb{C} \ \
(\forall \zeta \in \mathbb{R}) \  \text{ with } \ |\zeta| \le 1 \\
\text{implies} \ \  y = 0 ;
\end{eqnarray*}
(ii) complex uniformly convex (uniformly convex) if for every $\varepsilon > 0$ there
exists $\delta > 0$ such that
\begin{eqnarray*}
x, y \in Y, \ \|y\| \ge \varepsilon \ \text{ and } \ \|x + \zeta y\| \le 1, \ \ \forall \zeta \in \mathbb{C} \ \
(\forall \zeta \in \mathbb{R}) \  \text{ with } \ |\zeta| \le 1 \\
\text{implies} \ \ \|x\| \leq 1 - \delta .
\end{eqnarray*}
\end{definition}
It is clear that uniform convexity implies both complex uniform convexity and
strict convexity, while each of these two properties implies complex strict convexity.
Hilbert spaces and the $L_p$ spaces with $1 < p < \infty$ are uniformly convex (\cite{Cla},
see also \cite[Ch. III, $\S$1]{Die} or \cite[Theorem 11.10]{Car}).
$L_1$ is complex uniformly convex (see \cite{G1}) but not strictly convex.
$L_\infty$ does not have any of the above properties, but $(L_\infty)^*$ is
complex uniformly convex. Indeed, this space is isometrically isomorphic to
a space of bounded finitely additive set functions (see
\cite[Ch. IV, $\S$8, Theorem 16 and Ch. III, $\S$1, Lemma 5]{DS}) which is
complex uniformly convex (see \cite{Les}). Hence the class of spaces to which Theorem
\ref{strictlyA} applies includes
Hilbert spaces and $L_p(S, \Sigma, \mu)$ with $1 \le p \le \infty$, where
$(S, \Sigma, \mu)$ is an arbitrary measure space.

If $1\leq p<\infty$, then the $p$-direct sum
$X\oplus_pY$ of Banach spaces $X$ and $Y$ is the algebraic direct sum $X\oplus Y$ endowed with the norm
$$
\|(x,y)\|_p=\left(\|x\|_X^p+\|y\|_Y^p\right)^{1/p}.
$$
Similarly the $\infty$-direct sum $X\oplus_\infty Y$ is $X\oplus Y$ with the norm
$$
\|(x,y)\|_\infty=\max\left\{\|x\|_X,\|y\|_Y\right\}.
$$

These definitions are special cases of the absolute norms that we describe next.
Following \cite{BSW}, \cite[$\S$21]{BD} we say that a norm $\|\cdot\|$ on $\mathbb{C}^2$ is \textit{absolute} if
$$
\|(z, w)\| = \|(|z|, |w|)\| , \ \ \ \forall (z, w) \in \mathbb{C}^2
$$
and \textit{normalized} if
$$
\|(1, 0)\| = \|(0, 1)\| = 1 .
$$
Let $N_a$ denote the set of all absolute normalized norms on $\mathbb{C}^2$.

Let $\Psi$ denote the set of all continuous, convex functions $\psi$ on $[0, 1]$ such that $\psi(0) = \psi(1) = 1$ and
\begin{equation}\label{lowbd}
\max\{1 - t, t\} \le \psi(t) \le 1 , \ \ \ 0 \le t \le 1 . \end{equation}

\begin{theorem}\label{Na-Psi}(see \cite[$\S$21]{BD}).\newline
The formula $\|\cdot\| \to \psi(t)\equiv\|(1-t,t)\|$ defines a one-one map from $N_a$ onto $\Psi$ with inverse
$\psi\to \|\cdot\|_\psi$ given by
\begin{equation}\label{normpsi}
\|(z, v)\|_\psi := \left\{ \begin{array}{cl}
(|z| + |v|)\, \psi\left(\frac{|v|}{|z| + |v|}\right)\,  , & \   (z, v) \not= (0, 0) , \\
0 , &   \   (z, v) = (0, 0) .
\end{array} \right.
\end{equation}
If $\psi(t) \equiv \max\{1 - t, t\}$, then $\|\cdot\|_\psi$ coincides with the $l_\infty$ norm, while if
$\psi(t) \equiv 1$, then $\|\cdot\|_\psi$ coincides with the $l_1$ norm. Moreover
\begin{equation}\label{absnorms}
\frac12\, \|x\|_1 \le \|x\|_\infty \le \|x\|_\psi \le \|x\|_1 \le 2 \|x\|_\infty
\end{equation}
for all $\psi\in\Psi$ and all $x\in\mathbb{C}^2$.
\end{theorem}

Let $\psi \in \Psi$. Then the $\psi$-direct sum
$X\oplus_\psi Y$ of the Banach spaces $X$ and $Y$ is the space $X\oplus Y$
equipped with the norm
$$
\|(x,y)\|=\left\|(\|x\|_X, \|y\|_Y)\right\|_\psi .
$$
Let
\begin{equation}\label{P01}
P_0 : X\oplus_\psi Y \to X, \ \ P_1 : X\oplus_\psi Y \to Y
\end{equation}
be the canonical projections, and let
\begin{equation}\label{theta}
\theta(u) := \frac{\|P_0 u\|_X}{\|P_0 u\|_X + \|P_1 u\|_Y} = \frac{\|x\|_X}{\|x\|_X + \|y\|_Y}\, , \ \ \ u = (x, y) \in X\oplus_\psi Y .
\end{equation}
Then
\begin{eqnarray}
&& \|u\| = \left(\|P_0 u\|_X + \|P_1 u\|_Y\right) \psi\left(\frac{\|P_1 u\|_Y}{\|P_0 u\|_X + \|P_1 u\|_Y}\right) \nonumber \\
&& = \frac{1}{\theta(u)}\, \|P_0 u\|_X \psi\left(\frac{\|P_1 u\|_Y}{\|P_0 u\|_X + \|P_1 u\|_Y}\right) \le
\frac{1}{\theta(u)}\, \|P_0 u\|_X , \label{th1} \\ \nonumber \\
&& \|P_0 u\|_X = \theta(u) \left(\|P_0 u\|_X + \|P_1 u\|_Y\right) \le 2 \theta(u) \|u\|  \label{th2}
\end{eqnarray}
(see \eqref{absnorms}).

The space $X\oplus_\psi Y$ is complex uniformly (strictly) convex if and only if $X$ and $Y$ are
complex uniformly (strictly) convex and
\begin{equation}\label{cc}
\psi(t) > \max\{1 - t, t\} , \ \ \ \forall t \in (0, 1)
\end{equation}
(see \cite{DT}). It is interesting to compare this result to its real valued counterpart: the space $X\oplus_\psi Y$ is
uniformly (strictly) convex if and only if $X$ and $Y$ are
uniformly (strictly) convex and $\psi$ is strictly convex, i.e.
$$
s, t \in [0, 1] , \ s \not= t, \ 0 < c < 1 \ \ \Longrightarrow \ \
\psi((1 - c)s + c t) < (1 - c) \psi(s) + c \psi(t)
$$
(see \cite{KST, SK, TKS}).

Suppose $\psi(t_0) = 1 - t_0$ for some $t_0 \in (0, 1/2]$. Then the equality $\psi(0) = 1$, \eqref{lowbd} and the definition
of a convex function imply that $\psi(t) = 1 - t$ for all $t \in [0, t_0]$. Similarly, if $\psi(t_1) = t_1$ for some $t_1 \in [1/2, 1)$,
then $\psi(t) = t$ for all $t \in [t_1, 1]$.

The dual of $\|\cdot\|_\psi$ is the absolute normalized norm $\|\cdot\|_{\psi^*}$ with
\begin{equation}
\psi^*(t) := \max_{0 \le s \le 1} \frac{(1 - t)(1 - s) + t s}{\psi(s)}\label{psidualeq}
\end{equation}
(see \cite{BSW}, \cite{LF} or \cite[Theorem 5.4.19]{HJ}). If $\psi(s_0) = 1 - s_0$ for some $s_0 \in (0, 1/2]$, then
$$
\psi^*(t) \ge \frac{(1 - t)(1 - s_0) + t s_0}{1 - s_0} = (1 - t) + \frac{s_0}{1 - s_0}\, t > 1 - t  , \ \ \ \forall t \in (0, 1] .
$$
Similarly,  if $\psi(s_1) = s_1$ for some $s_1 \in [1/2, 1)$, then
\begin{equation}\label{t}
\psi^*(t) > t , \ \ \ \forall t \in [0, 1) .
\end{equation}

\section{Kolmogorov-Kallman-Rota type inequalities}

The main result in this Appendix, Theorem~\ref{Rota}, is used in Part 5 of the proof of Theorem~\ref{co1a}.
It is an extension to generators of arbitrary $C_0$ semigroups of a well known estimate for generators of contraction
semigroups (see \cite{CK}). Considering $B - \mu I$ with a sufficiently large $\mu \ge 0$ instead of the original operator $B$,
one can reduce the more general case to an estimate for the generator of a bounded
semigroup (see \eqref{KLanKol}). In order to derive \eqref{Dkn} from the latter, one needs estimates for intermediate powers of $B$,
and these are provided by Lemma \ref{eps}. Theorem~\ref{Rota*}, which extends Theorem~\ref{Rota} to the adjoints of
semigroup generators, is used in Part 7 of the proof of Theorem~\ref{co1a}, while the result in Remark~\ref{semirem} is used
in Part 2 of the proof.

We start with the following Landau-Kolmogorov type inequality for $n$ times continuously differentiable functions
$f : [0, \infty) \to \mathbb{C}$:
\begin{equation}\label{LanKol}
\left\|f^{(k)}\right\|^n_\infty \le M_{n, k} \|f\|^{n - k}_\infty \left\|f^{(n)}\right\|^k_\infty ,
\end{equation}
where $n \ge 2$,  $k = 1, \dots, n -1$, the constant $M_{n, k} < +\infty$ does not depend on $f$,
and $\|\cdot\|_\infty$ denotes the sup norm (see \cite{Kol, SC}).

Let $B_0$ be the infinitesimal generator of a $C_0$ semigroup $(T(t))_{t \ge 0}$ on a Banach space $X$ such that
$\|T(t)\| \le K$, $\forall t \ge 0$. Suppose $w \in \mbox{\rm Dom}(B_0^n)$. Then for any $g \in X^*$ with $\|g\|_{X^*} = 1$, the function
$f(t) := g(T(t)w)$ is $n$ times continuously differentiable  and
$$
f^{(m)}(t) = g(T(t) B^m_0w) , \ \ \ t \ge 0, \ \ m = 1, \dots, n
$$
(see, e.g., \cite[Lemmata 6.1.11 and 6.1.13]{Da}). Take $g$ such that $g(B^k_0w) = \|B^k_0w\|$.
Then \eqref{LanKol} implies
\begin{eqnarray}\label{KLanKol}
\|B^k_0w\|^n &=&  \left(g(B^k_0w)\right)^n \le \left(\sup_{t \ge 0} \left|g(T(t) B^k_0w)\right|\right)^n = \left\|f^{(k)}\right\|^n_\infty \nonumber \\
&\le& M_{n, k} \|f\|^{n - k}_\infty \left\|f^{(n)}\right\|^k_\infty  \nonumber \\
&=& M_{n, k} \left(\sup_{t \ge 0} \left|g(T(t) w)\right|\right)^{n - k}
\left(\sup_{t \ge 0} \left|g(T(t) B^n_0w)\right|\right)^k \\
&\le& M_{n, k} \left(\sup_{t \ge 0} \left\|T(t) w\right\|\right)^{n - k}
\left(\sup_{t \ge 0} \left\|T(t) B^n_0w\right\|\right)^k \nonumber \\
&\le& M_{n, k} K^n \|w\|^{n - k}\|B^n_0w\|^k , \ \ \ n \ge 2 , \ k = 1, \dots, n -1  \nonumber
\end{eqnarray}
(cf. \cite{CK}).
The optimal values of the constants $M_{n, k}$ are discussed in \cite{SC} and \cite{CK}. In particular, if
$n  = 2$, $k = 1$, then $M_{n, k} = M_{2, 1} = 4$, \eqref{LanKol} is the Landau inequality (\cite{Lan})
$$
\|f'\|^2_\infty \le 4 \|f\|_\infty\|f''\|_\infty ,
$$
and \eqref{KLanKol} is the Kallman-Rota inequality
$$
\|B_0w\|^2 \le 4K^2 \|w\| \|B^2_0w\|
$$
(see \cite{KR} or \cite[Chapter 1, Lemma 2.8]{P}).

Using the Kallman-Rota inequality one can easily show that if $B$ is the infinitesimal generator of a (possibly unbounded) $C_0$
semigroup on a Banach space $X$, then
there exists a constant $C > 0$ such that
\begin{equation}\label{21}
\|Bw\|^2 \le C (\|B^2w\| + \|w\|)\|w\| ,  \ \ \ \forall w \in \mbox{\rm Dom}(B^2)
\end{equation}
(see \cite{S3}). The proof is an almost trivial special case of the arguments given below.

It follows from \eqref{21} that for any $\varepsilon > 0$ there exists $C(\varepsilon) > 0$ such that
\begin{equation}\label{12}
\|Bw\| \le \varepsilon \|B^2w\| + C(\varepsilon)\|w\| ,  \ \ \ \forall w \in \mbox{\rm Dom}(B^2) .
\end{equation}

\begin{lemma}\label{eps}
For any $n \in \mathbb{N}$, $n \ge 2$, any $k = 1, \dots, n -1$, and any $\varepsilon > 0$ there exists $C_{n, k}(\varepsilon) > 0$ such that
\begin{equation}\label{kn}
\|B^kw\| \le \varepsilon \|B^nw\| + C_{n, k}(\varepsilon)\|w\| ,  \ \ \ \forall w \in \mbox{\rm Dom}(B^n) .
\end{equation}
\end{lemma}
\begin{proof}
The proof is by induction in $n$. The statement holds for $n = 2$ (see \eqref{12}). Suppose it holds for $n$.
Substituting $w = Bu$ into \eqref{kn} with $k = 1$ and using \eqref{12}, we get
\begin{eqnarray*}
&& \|B^2 u\| \le \frac\varepsilon2\, \|B^{n + 1}u\| + C_{n, 1}\left(\frac\varepsilon2\right)\|Bu\| \\
&& \le \frac\varepsilon2\, \|B^{n + 1}u\| + C_{n, 1}\left(\frac\varepsilon2\right)\left(\rho \|B^2u\| + C(\rho)\|u\|\right)  , \ \ \
\forall u \in \mbox{\rm Dom}(B^{n + 1}) .
\end{eqnarray*}
Taking $\rho = \frac12\, C_{n, 1}\left(\frac\varepsilon2\right)^{-1}$, we get
\begin{equation}\label{2n}
\|B^2 u\| \le \varepsilon \|B^{n + 1}u\| + C_{n + 1, 2}(\varepsilon) \|u\| , \ \ \
\forall u \in \mbox{\rm Dom}(B^{n + 1})
\end{equation}
with $C_{n + 1, 2}(\varepsilon) = 2 C_{n, 1}\left(\frac\varepsilon2\right) C\left(\frac12\, C_{n, 1}\left(\frac\varepsilon2\right)^{-1}\right)$.
Using \eqref{12} again, and then applying \eqref{2n} with $\varepsilon = 1$, we obtain
\begin{eqnarray}\label{1n}
\|Bu\| \le \varepsilon \|B^2u\| + C(\varepsilon)\|u\| \le \varepsilon \|B^{n + 1}u\| + C_{n + 1, 1}(\varepsilon) \|u\| , \\
\forall u \in \mbox{\rm Dom}(B^{n + 1})  , \nonumber
\end{eqnarray}
where $C_{n + 1, 1}(\varepsilon) = \varepsilon C_{n + 1, 2}(1) + C(\varepsilon)$.

Substituting $w = Bu$ into \eqref{kn} again and using \eqref{1n}, we get
\begin{eqnarray*}
\|B^{k + 1} u\| &\le& \frac\varepsilon2\, \|B^{n + 1}u\| + C_{n, k}\left(\frac\varepsilon2\right)\|Bu\| \\
&\le& \frac\varepsilon2\, \|B^{n + 1}u\| + C_{n, k}\left(\frac\varepsilon2\right) \left(\rho \|B^{n + 1}u\| + C_{n + 1, 1}(\rho) \|u\|\right) .
\end{eqnarray*}
Taking $\rho = \frac\varepsilon2\, C_{n, k}\left(\frac\varepsilon2\right)^{-1}$ and denoting $k + 1 = m$, we obtain
$$
\|B^m u\| \le \varepsilon \|B^{n + 1}u\| + C_{n + 1, m}(\varepsilon) \|u\| , \ \ \
\forall u \in \mbox{\rm Dom}(B^{n + 1}) , \ \ m = 3, \dots, n
$$
with $C_{n + 1, m}(\varepsilon) = C_{n, m - 1}\left(\frac\varepsilon2\right)
C_{n + 1, 1}\left(\frac\varepsilon2\, C_{n, m - 1}\left(\frac\varepsilon2\right)^{-1}\right)$. Together with \eqref{2n}, \eqref{1n} this
completes the proof.
\end{proof}

\begin{theorem}\label{Rota}
Let $B$ be the infinitesimal generator of a $C_0$ semigroup on a Banach space $X$.
Then for any $n \in \mathbb{N}$, $n \ge 2$, and any $k = 1, \dots, n -1$, there exists $L_{n, k} > 0$ such that
\begin{equation}\label{Dkn}
\|B^kw\|^n \le  L_{n, k}\left(\|B^nw\| + \|w\|\right)^k \|w\|^{n - k} ,  \ \ \ \forall w \in \mbox{\rm Dom}(B^n) .
\end{equation}
\end{theorem}
\begin{proof}
There exist constants  $\mu \ge 0$ and $K \ge 1$ such that the $C_0$ semigroup $T(t)$
generated by $B - \mu I$ satisfies the inequality $\|T(t)\| \le K$, $\forall t \ge 0$ (see, e.g.,
\cite[Chapter 1, Theorem 2.2]{P}). Increasing $\mu$ if necessary, we can assume that
$\left\|(B - \mu I)^{-1}\right\| \le 1$ (see, e.g., \cite[Theorem 12.3.1]{HP}).

Applying \eqref{kn} with $\varepsilon = 1$ in the penultimate inequality below, we get the following from \eqref{KLanKol}
with $B_0 = B - \mu I$
\begin{eqnarray*}
&& \|B^kw\|^n = \left\|\sum_{j = 0}^k \left(\begin{array}{c} k  \\  j \end{array}\right)(B - \mu I)^{k - j} \mu^j w\right\|^n  \\
&& \le \left(\sum_{j = 0}^k \left(\begin{array}{c} k  \\  j \end{array}\right) \mu^j \left\|(B - \mu I)^{k - j} w\right\|\right)^n    \\
&& \le (k + 1)^{n - 1} \sum_{j = 0}^k \left(\begin{array}{c} k  \\  j \end{array}\right)^n \mu^{jn} \left\|(B - \mu I)^{k - j} w\right\|^n    \\
&& \le (k + 1)^{n - 1} K^n \sum_{j = 0}^k \left(\begin{array}{c} k  \\  j \end{array}\right)^n \mu^{jn} M_{n, k - j}  \|(B - \mu I)^nw\|^{k - j} \|w\|^{n - k + j} \\
&& \le (k + 1)^{n - 1} K^n \left(\sum_{j = 0}^k \left(\begin{array}{c} k  \\  j \end{array}\right)^n \mu^{jn} M_{n, k - j}\right)
\|(B - \mu I)^nw\|^{k} \|w\|^{n - k} \\
&& =: D_{n, k} \|(B - \mu I)^nw\|^{k} \|w\|^{n - k} =
D_{n, k} \left\|\sum_{l = 0}^n \left(\begin{array}{c} n  \\  l \end{array}\right)(-\mu)^l B^{n - l} w\right\|^{k} \|w\|^{n - k} \\
&& \le D_{n, k} (n + 1)^{k - 1} \sum_{l = 0}^n \left(\begin{array}{c} n  \\  l \end{array}\right)^k \mu^{lk}\|B^{n - l} w\|^{k} \|w\|^{n - k}\\
&& \le D_{n, k} (n + 1)^{k - 1} \sum_{l = 0}^n \left(\begin{array}{c} n  \\  l \end{array}\right)^k \mu^{lk}
\left(\|B^nw\| + C_{n, n - l}(1)\|w\|\right)^k \|w\|^{n - k} \\
&& \le L_{n, k}\left(\|B^nw\| + \|w\|\right)^k \|w\|^{n - k} ,  \ \ \ \forall w \in \mbox{\rm Dom}(B^n)
\end{eqnarray*}
with a suitable constant $L_{n, k}$.
\end{proof}

If $X$ is reflexive, $B^*$ is the infinitesimal generator of a $C_0$ semigroup on $X^*$
(see, e.g., \cite[Corollary 3.3.9]{{ABHN}}), and \eqref{Dkn} holds with $B^*$ in place of $B$. The latter is true
even if $X$ is not reflexive, although $B^*$ might not be densely defined and the adjoint semigroup
might not be strongly continuous in this case.

\begin{theorem}\label{Rota*}
Let $B$ be the infinitesimal generator of a $C_0$ semigroup on a Banach space $X$.
Then for any $n \in \mathbb{N}$, $n \ge 2$, and any $k = 1, \dots, n -1$, there exists $L_{n, k} > 0$ such that
\begin{equation}\label{Dkn*}
\|(B^*)^kg\|^n \le  L_{n, k}\left(\|(B^*)^ng\| + \|g\|\right)^k \|g\|^{n - k} ,  \ \ \ \forall g \in \mbox{\rm Dom}((B^*)^n) .
\end{equation}
\end{theorem}

\begin{proof}
We start by proving an analogue of \eqref{KLanKol} for the adjoint $B^*_0$ of the infinitesimal generator $B_0$ of a $C_0$ semigroup $(T(t))_{t \ge 0}$ on $X$  such that
$\|T(t)\| \le K$, $\forall t \ge 0$.
The adjoint semigroup $(T^*(t))_{t \ge 0}$
is  weak* continuous, $B^*$ is its weak* infinitesimal generator, $\|T^*(t)\| \le K$, $\forall t \ge 0$,
the function
$f_x(t) := \left(T^*(t)g\right)x$ is $n$ times continuously differentiable for any $g \in \mbox{\rm Dom}((B_0^*)^n)$ and
any $x \in X$ with $\|x\| = 1$, and
$$
f_x^{(m)}(t) = \left(T^*(t) (B_0^*)^m g\right)x , \ \ \ t \ge 0, \ \ m = 1, \dots, n
$$
(see \cite[Proposition 1.4.4 and Corollary 1.4.5]{BB}). Then, as for \eqref{KLanKol}, one derives from \eqref{LanKol}
\begin{eqnarray}\label{LanKol*}
&& \|(B^*_0)^k g\|^n =  \left(\sup_{\|x\| = 1}\left|\left((B_0^*)^k g\right)x\right|\right)^n \le
\left(\sup_{\|x\| = 1}\sup_{t \ge 0} \left|\left(T^*(t) (B_0^*)^k g\right)x\right|\right)^n  \nonumber \\
 && = \sup_{\|x\| = 1}\left\|f_x^{(k)}\right\|^n_\infty
\le M_{n, k} \sup_{\|x\| = 1}\|f_x\|^{n - k}_\infty \left\|f_x^{(n)}\right\|^k_\infty  \nonumber \\
&& = M_{n, k} \sup_{\|x\| = 1}\left(\sup_{t \ge 0} \left|\left(T^*(t) g\right)x\right|\right)^{n - k}
\left(\sup_{t \ge 0} \left|\left(T^*(t) (B_0^*)^n g\right)x\right|\right)^k \\
&& \le M_{n, k} \left(\sup_{t \ge 0} \left\|T^*(t) g\right\|\right)^{n - k}
\left(\sup_{t \ge 0} \left\|T^*(t) (B_0^*)^n g\right\|\right)^k \nonumber \\
&& \le M_{n, k} K^n \|g\|^{n - k}\| (B_0^*)^n g\|^k , \ \ \ n \ge 2 , \ k = 1, \dots, n -1 .  \nonumber
\end{eqnarray}

The proof is completed by using \eqref{LanKol*} instead of \eqref{KLanKol} and repeating the proof of Theorem \ref{Rota}.

\end{proof}

\begin{remark}\label{semirem}
It was shown in \cite{S3} that the resolvent norm of the
infinitesimal generator of a $C_0$ semigroup on a Banach space
cannot be constant on an open set if the underlying space is complex
uniformly convex. The proof relied on estimate \eqref{21}. Theorem
\ref{Rota*} allows one to extend the main result of \cite{S3} to the case where
the dual of the underlying Banach space, rather than the space itself,
is complex uniformly convex. The theorem by S. B\"ogli and P. Siegl (\cite{BS})
mentioned in the Introduction provides an easier way of proving this `dual' result.
\end{remark}

{\it Acknowledgements.} \
We are grateful to the anonymous referee for helpful comments.

\end{document}